\documentclass[11pt,a4paper]{article}

\usepackage{amsmath,amssymb,}
\usepackage{authblk}
\usepackage{amsthm}\usepackage{amsthm}

\newtheorem{thm}{Theorem}[section]
\newtheorem{lem}[thm]{Lemma}
\newtheorem{prop}[thm]{Proposition}

\newtheorem{rem}[thm]{Remark}

\begin{document}
\title{\huge \textbf{On unitary groups of crossed product von Neumann algebras}}
\author{Yasuhito Hashiba \thanks{Graduate School of Mathematics, University of Tokyo, Komaba, Tokyo 153-8914, Japan\\  email: \texttt{y.hashiba00@gmail.com}}}
\date{}
\maketitle

\begin{abstract}
We consider the tracial crossed product algebra $M=A\rtimes\Lambda$ arising from a trace preserving action $\sigma:\Lambda \curvearrowright A$ of a discrete group $\Lambda$ on a tracial von Neumann algebra $A$. For a unitary subgroup $\mathcal{G}\subset \mathcal{U}(M)$, we study when this $\mathcal{G}$ can be conjugated into $\mathcal{U}(A)\cdot\Lambda$ in $M$. We provide a general sufficient condition for this to happen. Our result generalizes  \cite[Theorem 3.1]{ioana2013class} which treats the case when $M$ is the group von Neumann algebra $L(\Lambda)$.
\end{abstract}

\section{INTRODUCTION} 
When we have a (tracial) von Neumann algebra $A$ and a discrete group $\Lambda$ which acts on $A$, we can construct the \textit{crossed product von Neumann algebra} $A\rtimes\Lambda$. In the case  $A=\mathbb{C}$, we obtain the \textit{group von Neumann algebra} $L(\Lambda)$. Many interesting and important examples of von Neumann algebras have been obtained through this construction. 

A problem of central importance in von Neumann algebra theory is to determine what kind of information of the initial action $\Lambda\curvearrowright A$ can be recovered from the crossed product  $A\rtimes \Lambda$. In general, a large part of the information on the action may fade away when passing to the crossed product algebra. This is best illustrated in the \textit{amenable} case. By the celebrated result of Connes \cite{connes1976classification} which shows that the \textit{hyperfinite}  $\mathrm{II_{1}}$ factor is the unique amenable  $\mathrm{II_{1}}$ factor, we see that any crossed product $\mathrm{II_{1}}$ factor $A\rtimes \Lambda$ with $A$ and $\Lambda$ amenable is isomorphic to the hyperfinite $\mathrm{II_{1}}$ factor. In other words, amenability is the only information that can be recovered from the crossed product algebra in this setting. 

However, the situation in the non-amenable case is significantly different. Thanks to the progress in Popa's deformation/rigidity theory  (\cite{popa2003strrgd}) during the past two decades, a large class of  actions $\Lambda\curvearrowright A$ that can be (partially) recovered from the crossed product  $A\rtimes \Lambda$ have been found. See for example \cite {ioana2011w}, \cite{ioana2013class}, \cite{popa2014unique} and the surveys \cite{vaes2010rigidity}, \cite{ioana2018rigidity} for an overview of recent results.

For many of the results in this direction, the key point is to show unitary conjugacy of two subalgebras (for example, $A$ and $B$ in $A\rtimes \Lambda = B\rtimes \Gamma$). A powerful method for showing such a result is the famous \textit{intertwining-by-bimodules technique} by Popa (\cite{popa2003strrgd}), which is repeatedly used in the above papers.

On the other hand, in the same situation $A\rtimes\Lambda = B\rtimes \Gamma$, determining if the canonical unitaries $\{u_{g}\}_{\in \Gamma}$ and $\{v_{h}\}_{h\in \Lambda}$ are unitarily conjugate is also an important and natural problem to study in the context of crossed product algebras. In this direction, Ioana, Popa and Vaes have found an interesting criterion in the setting of group von Neumann algebras  \cite[Theorem 3.1]{ioana2013class}. This theorem says that if we have $L(\Gamma)=L(\Lambda)$ for two i.c.c. groups $\Gamma$ and $\Lambda$, and $\inf_{g\in \Gamma} h_{\Lambda}(u_{g})>0$, where $h_{\Lambda}(x)=\max\{|\tau(xv_{h}^{*})| |h\in\Lambda\}$, then  $\mathbb{T}\Gamma$ and $\mathbb{T}\Lambda$ are unitarily conjugate. In this paper, we generalize this criterion to arbitrary crossed products. Our main result is the following theorem.

\begin{thm}
Let $M=A\rtimes\Lambda= B\times\Gamma$ be two crossed product decompositions of a  $\mathrm{II}_{1}$ factor $M$ with $\Lambda$ an i.c.c. group and assume that the action of $\Lambda$ on $L^{2}(Z(A))\ominus \mathbb{C}$ is weakly mixing. Set  $\mathcal{G}=\{bu_{g}|b\in \mathcal{U}(B), g\in \Gamma\}  \subset \mathcal{U}(M)$ and assume that the following conditions hold.
\begin{enumerate}
\item  The unitary representation $\{\mathrm{Ad} x\}_{x \in \mathcal{G}}$ on $L^{2}(M)\ominus \mathbb{C}$ is weakly mixing.
\item $\inf_{x\in \mathcal{G}}\sup_{h\in\Lambda}\|E_{A}(xv_{h}^{*})\|_{2} >0$.
\end{enumerate}
 Then there exists a unitary $w\in \mathcal{U}(M)$ such that $w\mathcal{G}w^{*}\subset \{av_{h}| a\in \mathcal{U}(A), h \in \Lambda\}$ holds.  \label{mainthm}
\end{thm}

See Theorem \ref{theorem} for a more general version. This can be also thought of as an ``$\Gamma$-acting version'' of \cite[Theorem 1.3.2] {ioana2011w}.

As an application of our methods in this paper, we  give a short proof of an important result from \cite{popa2007cocycle} on unitary cocycles..

\section{PRELIMINARIES}  
\subsection{Terminology}
A tracial von Neumann algebra $(M, \tau)$ is a von Neumann algebras $M$ endowed with a faithful normal tracial state $\tau : M \to \mathbb{C}$. We denote by $L^{2}(M)$ the Hilbert space obtained by the GNS construction of $(M,\tau)$ and $\|x\|_{2} = \tau(x^{*}x)^{1/2}$ the norm of $x\in M$ considered as an element of $L^{2}(M)$. We say that $M$ is separable if it is separable with respect to this norm. Every von Neumann algebra considered in this paper is assumed to be  separable.
The adjoint map $x \in M\mapsto x^{*}\in M$ extends to an  anti-linear isometric involution on $ L^{2}(M)$ and is denoted as $J:L^{2}(M)\to L^{2}(M)$.  The group of unitary elements  of $M$ is denoted by $\mathcal{U}(M)$  and the set of projections of $M$ is denoted by $P(M)$. We set $(M)_{r}= \{x\in M | \parallel x \parallel \le r \}$ for $r\ge 0$. For a von Neumann subalgebra $N \subset M$, we regard $N$ as a tracial von Neumann algebra together with $\tau |_{N}$, and  $E_{N}$ denotes the trace preserving normal conditional expectation $M \to N$ and $e_{N}:L^{2}(M)\to L^{2}(N)$ the corresponding projection. For $n\in \mathbb{N}$, we write $M_{n}(M)$ for $M_{n}(\mathbb{C})\otimes M$. We also write  $M_{\infty}(M)$ for $B(l^{2}(\mathbb{N}))\bar{\otimes}M$. The canonical matrix units for $B(l^{2}(\mathbb{N}))$ are denoted as  $\{e_{i,j}\}_{i,j\in\mathbb{N}}\in B(l^{2}(\mathbb{N}))$. When $M_{n}(M)$ is in consideration for some $n \in \mathbb{N}\cup \{\infty\}$, and $N\subset M$ is a subalgebra, $E_{N}$ also denotes the amplified conditional expectation $M_{n}(M)\to M_{n}(N)$.  For elements in a tensor product von Neumann algebra, we use the tensor leg numbering notation. For example, we write 
\begin{align*}
&(a\otimes b)_{13} = a\otimes 1 \otimes b \in M\bar{\otimes}M\bar{\otimes}M, \\
&(a\otimes b\otimes c)_{134}=a\otimes 1 \otimes b \otimes c\in M\bar{\otimes}M\bar{\otimes}M\bar{\otimes}M, \\
&(a\otimes b\otimes c)_{132}=a\otimes c \otimes b  \in M\bar{\otimes}M\bar{\otimes}M,
\end{align*}
for $a,b,c\in M$. We extend this notation to amplifications, so that for $e_{i,j}\otimes a\otimes b \in  M_{\infty}(M\bar{\otimes}M)$, we have \[(e_{i,j}\otimes a\otimes b)_{13} = e_{i,j}\otimes a\otimes 1 \otimes b \in M_{\infty}(M\bar{\otimes}M\bar{\otimes}M)\] for  $a,b\in M$ and $i,j \in \mathbb{N}$.

 We refer to \cite{anantharaman2017introduction} for the fundamental facts about (tracial) von Neumann algebras used in this paper.

\subsection{Measurable operators}
Fix a faithful semi-finite normal tracial weight $\tau$ on a semi-finite von Neumann algebra $M$ and consider the Hilbert space $L^{2}(M.\tau)$ obtained by the GNS construction. An operator $A$ on  $L^{2}(M.\tau)$ is called $\tau$-measurable if it is a densely defined closed operator affiliated with $M$ and the spectral decomposition  $ |A|=\int_{0}^{\infty}\lambda \mathrm{d}e(\lambda)$ of $|A|$ satisfies $\lim_{\lambda \to \infty}\tau(e(\lambda)^{\bot})=0$. The space of $\tau$-measurable operators forms a ${*}$-algebra, and we will freely take adjoints, sums and products of $\tau$-measurable operators. Here, the sum  $A+B$ (respectively the product $AB$) of two $\tau$-measurable operators $A,B$ is the closure of the algebraic sum (respectively the product). Note that elements of $L^{2}(M.\tau)$ can canonically be identified with a  $\tau$-measurable operator. We refer to \cite[Chapter IX, $\S 2$]{takesaki2013theory} for these facts. We use the facts in this section especially when $M=M_{\infty}(N)$ and $\tau=\mathrm{Tr}\otimes \tau_{N}$ where $(N,\tau_{N})$ is a tracial von Neumann algebra and $\mathrm{Tr}$ is the usual tracial weight on $B(l^{2}(\mathbb{N}))$.

\subsection{Popa's intertwining technique}
In \cite{popa2003strrgd}, Popa introduced a powerful criterion to identify intertwiners between subalgebras of tracial von Neumann algebras which is now called Popa's intertwining-by-bimodules technique.
\begin{thm}(\cite{popa2003strrgd})
Let M be a tracial von Neumann algebra and $P,Q\subset M$ be two subalgebras. Then the following conditions are equivalent.
\begin{enumerate}
\item For any subgroup $\mathcal{G}\subset \mathcal{U}(P)$ with $\mathcal{G}''=P$, there is no net $(u_{n})_{n}\subset \mathcal{G}$ satisfying $\|E_{Q}(xu_{n}y)\|_{2} \to 0$ for all $x,y\in M$.
\item There exist non-zero projections $p\in P(P),q\in P(Q)$, a ${*}$-homomorphism $\theta:pPp\to qQq$ and a non-zero partial isometry $v\in pMq$ such that $xv=v\theta(x)$ holds for all $x\in pPp$.
\end{enumerate}
If the above conditions are satisfied, we say that a corner of $P$ embeds into $Q$ inside $M$, and write $P\prec_{M}Q$.
\end{thm}

\subsection{Crossed products}
Given a trace preserving action of a countable discrete group $\Lambda$ on a tracial von Neumann algebra $(A,\tau_{A})$, we denote by $A\rtimes\Lambda$ the corresponding crossed product von Neumann algebra, which is tracial together with trace defined by $\tau(av_{h})=\delta_{e,h}\tau_{A}(a)$ for $a\in A$ and $h\in \Lambda$. Every crossed product algebra considered in this paper comes in this form (in particular, every action of a group on a tracial von Neumann algebra is assumed to be trace preserving).  For a crossed product von Neumann algebra  $M=A\rtimes\Lambda$, we define the ${*}$-homomorphism $\Delta:M\to L(\Lambda)\bar{\otimes}M$ by $\Delta(av_{h})=v_{h}\otimes av_{h}$ for $a \in A$ and $h \in \Lambda$. We also consider the basic construction $(\langle M,e_{A} \rangle,\hat{\tau})$. By definition, $\langle M,e_{A} \rangle$ is the commutant of $JAJ$ on $L^{2}(M)$ and $\hat{\tau}$ is the canonical tracial weight on $\langle M,e_{A} \rangle$ satisfying $\hat{\tau}(xe_{A}y)=\tau(xy)$ for $x,y\in M$.  The basic construction comes with a canonical operator valued weight $E$ from  $\langle M,e_{A} \rangle^{+}$ to the extended positive part of $M$ which satisfies $E(xe_{A}y)=xy$ for  $x,y\in M$. Since $\Sigma_{h\in \Lambda}Jv_{h}Je_{A}Jv_{h}^{*}J=1$, we have $E(x)=\Sigma_{h\in \Lambda}Jv_{h}JxJv_{h}^{*}J$ for $x\in \langle M,e_{A} \rangle^{+}$. We end this section with a lemma from  \cite[Proposition 7.2]{ioana2013class}. For a group $\Lambda$, we write $C_{\Lambda}(h)=\{k\in\Lambda|kh=hk\}$ for $h\in\Lambda$.

\begin{lem}
Let  $M=A\rtimes\Lambda$ be a crossed product obtained from a countable discrete group $\Lambda$ acting on $A$ and $\Delta :M\to L(\Lambda)\bar{\otimes}M$ defined as above. Let $B\subset M$ be a subalgebra and assume that $B \nprec_{M} A\rtimes C_{\Lambda}(h)$ holds for all $h\in \Lambda \setminus \{e\}$. Then, for any $\Delta(B)$-$\Delta(M)$-subbimodule $H\subset L^{2}( L(\Lambda)\bar{\otimes}M)$ that is finitely generated as a right $\Delta(M)$-module, we have $H\subset L^{2}(\Delta(M))$. In particular, for any subgroup $\mathcal{G}\subset \mathcal{U}(B)$ with $\mathcal{G}''=B$, the unitary representation $\{\mathrm{Ad}\Delta(x)\}_{x\in\mathcal{G}} $ on $L^{2} (L(\Lambda)\bar{\otimes}M)\ominus L^{2}(\Delta(M))$ is weakly mixing. \label{lem0}
\end{lem}
\begin{proof} From our assumption, we can take a sequence $\{u_{n}\}$ from $\mathcal{U}(B)$ such that $\|E_{A\rtimes C_{\Lambda}(h)}(zu_{n}w)\|_{2}\to 0$ holds for all $z,w\in M$ and $h\in \Lambda \setminus \{e\}$. It suffices to show that $\|E_{\Delta(M)}(x\Delta(u_{n})y)\|_{2}\to 0$ for all $x,y\in (L(\Lambda)\bar{\otimes}M)\ominus \Delta(M)$. We may moreover assume that $x,y$ are of the form $x=v_{k}\otimes 1,y=v_{l}\otimes 1$ with $k,l\in \Lambda \setminus \{e\}$. Writing $u_{n}=\Sigma_{m\in\Lambda}a_{m}v_{m}$, we have 
\begin{align*}
\|E_{\Delta(M)}(x\Delta(u_{n})y)\|_{2}^{2} &=\|\Sigma_{m\in\Lambda}E_{\Delta(M)}(v_{k}v_{m}v_{l}\otimes a_{m}v_{m})\|_{2}^{2} \\
&= \Sigma_{m\in\Lambda,kml=m}\|a_{m}\|_{2}^{2}.
\end{align*}
If there is no element $g\in \Lambda$ with $kgl=g$, the last expression in zero. If there is an element $g\in \Lambda$ with $kgl=g$, then for any $m\in \Lambda$ satisfying $kml=m$, we have $kmg^{-1}k^{-1}=mg^{-1}$ and thus $mg^{-1}\in C_{\Lambda}(k)$. Hence the last expression equals 
\begin{align*}
\Sigma_{m\in C_{\Lambda}(k)g}\|a_{m}\|_{2}^{2} = \|E_{A\rtimes C_{\Lambda}(k)}(u_{n}v_{g}^{*})\|_{2}^{2},
\end{align*}
which tends to zero since $k\neq e$.
\end{proof}

\section{PROOF OF THE MAIN THEOREM} 
We start with an elementary lemma in a formulation suited for our applications.
\begin{lem}
Let $N,M$ be two tracial von Neumann algebras and set $H= l^{2}(\mathbb{N})\otimes L^{2}(N)$.
\begin{enumerate}
\item The space of  Hilbert-Schmidt operators from $L^{2}(M)$ to $H$ (denoted as $B^{2}(L^{2}(M),H)$) is naturally isomorphic to $H\otimes L^{2}(M)$ through the following map:
\begin{align*}
\xi\otimes \eta \in H\otimes L^{2}(M) \mapsto [v\in L^{2}(M)\mapsto \langle \eta^{*}, v \rangle \xi \in H] \in B^{2}(L^{2}(M),H).
\end{align*}
\item Let $W\in H\otimes L^{2}(M) = l^{2}(\mathbb{N})\otimes L^{2}(N\bar{\otimes}M)$ and $V\in B^{2}(L^{2}(M),H)$ correspond to each other through the above isomorphism. Also view $W$ as an element of $L^{2}(M_{\infty}(N\bar{\otimes}M),\mathrm{Tr}\otimes\tau)$ (and thus as an $(\mathrm{Tr}\otimes \tau)$-measurable operator associated to $M_{\infty}(N\bar{\otimes}M)$) through the following isomorphism
\begin{align*}
\delta_{i}\otimes\eta\otimes\zeta\in l^{2}(\mathbb{N})\otimes L^{2}(N\bar{\otimes}M)\mapsto e_{i,0}\otimes \eta\otimes\zeta\in L^{2}(M_{\infty}(N\bar{\otimes}M),\mathrm{Tr}\otimes\tau)e_{0}.
\end{align*}
Then, identifying $N\bar{\otimes}M$ with $e_{0}M_{\infty}(N\bar{\otimes}M)e_{0}$, we have
\begin{itemize}
\item $\tau(W^{*}W)=\langle W,W\rangle_{H\otimes L^{2}(M)} =\mathrm{Tr}(V^{*}V) (<\infty)$,
\item $\langle y,V^{*}a^{op}Vz\rangle = \tau(W_{13}^{*}W_{12}(a\otimes z\otimes y^{*}))$,
\end{itemize}
for any $a\in N$ and $z,y\in M$ where elements of $N^{op}$ act on $H$ by natural right multiplication. 
\end{enumerate}\label{lem1}
\end{lem}
\begin{proof} We only prove $\langle y,V^{*}a^{op}Vz\rangle = \tau(W_{13}^{*}W_{12}(a\otimes z\otimes y^{*}))$ for  $a\in N$ and $z,y\in M$ (the other facts are well known and easy to verify). Let $W_{j}=\delta_{i_{j}}\otimes\eta_{j}\otimes\zeta_{j}$ for $j=1,2$ be two vectors in $ l^{2}(\mathbb{N})\otimes L^{2}(N\bar{\otimes}M)$ and $V_{j}\in B^{2}(L^{2}(M),H)$ be the corresponding  Hilbert-Schmidt operators. Then we have
\begin{align*}
\langle y,V_{1}^{*}a^{op}V_{2}z\rangle &= \langle V_{1}y,a^{op}V_{2}z\rangle \\
&= \langle \langle \zeta_{1}^{*},y\rangle \delta_{i_{1}}\otimes \eta_{1},\langle \zeta_{2}^{*},z\rangle \delta_{i_{2}}\otimes \eta_{2}a \rangle \\
&= \delta_{i_{1},i_{2}}\overline{\tau(\zeta_{1}y)}\tau(\zeta_{2}z)\tau(\eta_{1}^{*}\eta_{2}a) \\
&= \delta_{i_{1},i_{2}}\tau(\eta_{1}^{*}\eta_{2}a\otimes \zeta_{2}z\otimes y^{*}\zeta_{1}^{*}) \\
&= \delta_{i_{1},i_{2}}\tau((\eta_{1}^{*}\otimes 1\otimes \zeta_{1}^{*})(\eta_{2}\otimes\zeta_{2}\otimes 1)(a\otimes z\otimes y^{*})) \\
&= \tau((W_{1})_{13}^{*}(W_{2})_{12}(a\otimes z\otimes y^{*})).
\end{align*}
Since the linear span of vectors of the form $\delta_{i_{j}}\otimes\eta_{j}\otimes\zeta_{j}$ is dense in  $ l^{2}(\mathbb{N})\otimes L^{2}(N\bar{\otimes}M)$, we see that the above equation holds for any $W_{j}\in  l^{2}(\mathbb{N})\otimes L^{2}(N\bar{\otimes}M)$.
Thus, we obtain the second item.
\end{proof}

The following lemma gives the first step towards our main result.

\begin{lem}
Let $N, M$ be two tracial von Neumann algebras. For a unitary $Z\in \mathcal{U}(N\bar{\otimes}M\bar{\otimes}M )$, the following conditions are equivalent.
\begin{enumerate}
\item There exists a unitary $U\in\mathcal{U}(N\bar{\otimes}M)$ such that $Z=U_{13}^{*}U_{12}$ holds.
\item $Z_{134}Z_{123} = Z_{124}. $ 
\end{enumerate} \label{newlem}
\end{lem}
\begin{proof} 
We easily see that the first condition implies the second condition. To prove the converse, for each $a\in N$, consider the sesquilinear form $\sigma_{a}:(y,z) \in L^{2}(M)^{2}\mapsto \tau(Z (a\otimes z\otimes y^{*}))$. Since $|\tau(Z (a\otimes z\otimes y^{*})| \le  \|a\| \|y\|_{2} \|z\|_{2}$, the sesquilinear form $\sigma_{a}$ is bounded. Thus
for each  $a\in N$, there exists a unique operator $T_{a}\in B(L^{2}(M))$ such that $\sigma_{a}(y,z) =  \tau(Z (a\otimes z\otimes y^{*})) = \langle y,T_{a}z\rangle$ holds for every $y,z\in L^{2}(M)$. 

We now consider the map 
$T:a^{op}\in N^{op} \mapsto T_{a}\in B(L^{2}(M))$.
We prove that this map $T:N^{op} \to B(L^{2}(M))$ is a normal completely positive map. Take any $n\in \mathbb{N}$ and $[a_{ij}^{op}]_{ij} \in M_{n}(N^{op})_{+}$. We show that  $[T_{a_{ij}} ]_{ij}\in M_{n}(B(L^{2}(M)))$ is positive. For any $\xi=(\xi_{1},\dots, \xi_{n})\in L^{2}(M)^{n}$, we have 
\begin{align*}
\langle\xi,[T_{a_{ij}} ]_{ij} \xi\rangle &= \Sigma_{i,j} \langle\xi_{i},T_{a_{ij}} \xi_{j}\rangle \\
&= \Sigma_{i,j} \tau(Z(a_{ij}\otimes \xi_{j}\otimes\xi_{i}^{*})) \\
&= \Sigma_{i,j} \tau(Z_{134}^{*}Z_{124}(a_{ij}\otimes \xi_{j}\otimes\xi_{i}^{*}\otimes 1)),
\end{align*}
where we have used our assumption $Z_{134}Z_{123} = Z_{124}$ for the third equation. Setting \[\eta_{i} = (1\otimes\tau\otimes1)(Z(1\otimes \xi_{i} \otimes1)) \in L^{2}(N\bar{\otimes}M), \]
\[ \eta = (\eta_1, \dots, \eta_n) \in L^{2}(N\bar{\otimes}M)^{n}, \] we have 
\begin{align*}
\Sigma_{i,j} \tau(Z_{134}^{*}Z_{124}(a_{ij}\otimes \xi_{j}\otimes\xi_{i}^{*}\otimes 1)) &= \Sigma_{i,j}\tau(\eta_{i}^{*}\eta_{j}(a_{ij}\otimes1)) \\
&=\langle\eta,([a_{ij}^{op}]_{ij}\otimes1)\eta\rangle \ge 0,
\end{align*}
where elements of $N^{op}$ act naturally on $L^{2}(N)$ by right multiplication. This shows that $[T_{a_{ij}} ]_{ij}\in M_{n}(B(L^{2}(M)))$ is positive, and complete positivity of  $T:N^{op} \to B(L^{2}(M))$ follows. Normality of $T$ is obvious from the definition. We also have that $T_{1}\in B(L^{2}(M))$ is a trace class operator. Indeed, for a complete orthonormal basis $\{y_{i}\}_{\in I} \subset M$ of $L^{2}(M)$, we can compute
\begin{align*}
\Sigma_{i}\langle y_{i}, T_{1}y_{i}\rangle &= \Sigma_{i} \tau(Z(1\otimes y_{i}\otimes y_{i}^{*})) \\
&= \Sigma_{i} \tau(Z_{134}^{*}Z_{124}(1\otimes y_{i}\otimes y_{i}^{*}\otimes 1))\\
&= \Sigma_{i} \| (1\otimes \tau \otimes 1)(Z(1\otimes y_{i}\otimes 1))\|_{2}^{2}\\
&= \|Z\|_{2}^{2} = 1.
\end{align*}

We can now take the Stinespring dilation $(\pi,H,V)$ of $T$, that is, a Hilbert space $H$, a normal representation $\pi: N^{op}\to B(H)$, and a bounded linear operator $V:L^{2}(M)\to H$ such that $T_{a} = V^{*}\pi(a^{op})V$ holds for all $a\in  N$. Viewing $H$ as a right $ N$-module by $\pi$, we may assume that $H=l^{2}(\mathbb{N})\otimes L^{2}(N)$ and $\pi(a^{op})$ is the right multiplication of $a\in N$ on $l^{2}(\mathbb{N})\otimes  L^{2}(N)$, so hereafter we omit $\pi$ and write $T_{a}= V^{*}a^{op}V$. Since $T_{1} = V^{*}V$ is a trace class operator, we see that $V$ is a  Hilbert-Schmidt operator from $L^{2}(M)$ to $H$. 

Take \[W\in H\otimes L^{2}(M) = l^{2}(\mathbb{N})\otimes L^{2}(N\bar{\otimes}M)\cong  L^{2}(M_{\infty}(N\bar{\otimes}M))e_{0}\] which corresponds to $V\in B^{2}(L^{2}(M),H)$ through the isomorphism in Lemma \ref{lem1} and view $e_{0}M_{\infty}(N\bar{\otimes}M)e_{0} \cong N\bar{\otimes}M$ so that we have  $\tau(W^{*}W)=\mathrm{Tr}(V^{*}V)<\infty$. Using the same lemma, we have
\begin{align*}
\tau(Z (a\otimes z\otimes y^{*})) &= \langle y,T_{a}z\rangle \\
&=  \langle y,V^{*}a^{op}Vz\rangle \\
&= \tau(W_{13}^{*}W_{12}(a\otimes z\otimes y^{*}))
\end{align*}
for all $z,y\in M$ and $a\in N$. Thus we have \[Z=W_{13}^{*}W_{12}\in N\bar{\otimes}M\bar{\otimes}M\cong e_{0}M_{\infty}(N\bar{\otimes}M\bar{\otimes}M)e_{0}.\] Since we have $W_{12}^{*}W_{13}W_{13}^{*}W_{12}  = Z^{*}Z= e_{0}\in M_{\infty}(\mathbb{C})$, applying $E_{N\bar{\otimes}M\otimes 1}$ (the amplified conditional expectation  $M_{\infty}(N\bar{\otimes}M\bar{\otimes}M)\to M_{\infty}(N\bar{\otimes}M\otimes1)$ and the extension to the extended positive parts of the two algebras) to both sides, we obtain
$ W^{*}E_{N\otimes 1}(WW^{*})W =e_{0}$.  From this, we see that 
\[\tilde{W} = E_{N\otimes 1}(WW^{*})^{1/2}W\in M_{\infty}(N\bar{\otimes}M)\] is a partial isometry with $\tilde{W}^{*}\tilde{W}=e_{0}$.
Now, we have \[W_{14}^{*}W_{13}W_{13}^{*}W_{12}=Z_{134}Z_{123}=Z_{124},\]  and thus \[W_{14}^{*}E_{N\otimes 1}(WW^{*})_{12}W_{12} = Z_{124},\] which shows $\tilde{W}_{13}^{*}\tilde{W}_{12}= Z$. We now have \[\tilde{W}_{12}\tilde{W}_{12}^{*}\tilde{W}_{13}\tilde{W}_{13}^{*}\tilde{W}_{12}\tilde{W}_{12}^{*} = \tilde{W}_{12}\tilde{W}_{12}^{*}\in P( M_{\infty}(N)\bar{\otimes}M\bar{\otimes}M),\] which implies $\tilde{W}_{12}\tilde{W}_{12}^{*} \ge \tilde{W}_{13}\tilde{W}_{13}^{*}$. Flipping the second and third components of $N\bar{\otimes}M\bar{\otimes}M$, we also see that  $\tilde{W}_{12}\tilde{W}_{12}^{*} \le \tilde{W}_{13}\tilde{W}_{13}^{*}$ holds, so we have $\tilde{W}_{12}\tilde{W}_{12}^{*}=\tilde{W}_{13}\tilde{W}_{13}^{*}$. This shows $\tilde{W}\tilde{W}^{*}\in M_{\infty}(N)\bar{\otimes}1$, so we can write  $\tilde{W}\tilde{W}^{*}=e$ where $e\in P(M_{\infty}(N))$. Since $\tilde{W}^{*}\tilde{W}=e_{0}$ in $M_{\infty}(N\bar{\otimes}M)$, applying the center valued tracial weight, we have \[(\mathrm{Tr}\otimes E_{Z(N)})(e)=(\mathrm{Tr}\otimes E_{Z(N)}\otimes E_{Z(M)})(e)=(\mathrm{Tr}\otimes E_{Z(N)}\otimes E_{Z(M)})(e_{0})=(\mathrm{Tr}\otimes E_{Z(N)})(e_{0}),\]
which shows that  $e$ and $e_{0}$ are equivalent projections in $M_{\infty}(N)$. Hence we can take a partial isometry $v\in M_{\infty}(N)$ with $vv^{*}=e_{0}$ and $v^{*}v=e$. Setting $U=(v\otimes 1)\tilde{W}\in e_{0}M_{\infty}(N\bar{\otimes}M)e_{0} \cong N\bar{\otimes}M$, we have  that $U$ is a unitary and $Z=\tilde{W}_{13}^{*}\tilde{W}_{12}=\tilde{W}_{13}^{*}(v^{*}v\otimes1 \otimes1)\tilde{W}_{12}=U_{13}^{*}U_{12}$. 
\end{proof}

\begin{prop}Let $M=A\rtimes\Lambda$ be a crossed product $\mathrm{II}_{1}$ factor and $\mathcal{G} \subset \mathcal{U}(M)$ a subgroup of the unitary group of $M$. Assume that the following conditions hold. 
\begin{itemize}
\item $\inf_{x\in \mathcal{G}}\sup_{h\in\Lambda}\|E_{A}(xv_{h}^{*})\|_{2} >0$.
\item The unitary representation $\{\mathrm{Ad} x\}_{x \in \mathcal{G}}$ on $L^{2}(M)\ominus \mathbb{C}$ is weakly mixing.
\item The action $\{\mathrm{Ad}\Delta(x)\}_{x\in \mathcal{G}}$ on $L(\Lambda)\bar{\otimes}M $ is ergodic.
\end{itemize}
Then there exist a unitary $U\in \mathcal{U}(L(\Lambda)\bar{\otimes}M)$ and a unitary representation $\{w_{x}\}_{x\in \mathcal{G}}\subset \mathcal{U}(L(\Lambda))$ of $\mathcal{G}$ with $U\Delta(x)U^{*} = w_{x}\otimes x$ for all $x\in \mathcal{G}$.   \label{prop1}
\end{prop}
\begin{proof} Set $\delta = \inf_{x\in \mathcal{G}}\sup_{h\in\Lambda}\|E_{A}(xv_{h}^{*})\|_{2} >0$. Take any $x\in \mathcal{G}$ and write $x = \Sigma_{h\in \Lambda} a_{h}v_{h}$. Then we have 
\begin{align*}
\tau((\Delta(x)\otimes x)_{132}(\Delta(x)\otimes x)^{*}) 
&= \Sigma_{h,k,l,m\in \Lambda} \tau((v_{h}\otimes a_{k}v_{k}\otimes a_{h}v_{h})(v_{l}\otimes a_{l}v_{l}\otimes a_{m}v_{m})^{*}) \\
&= \Sigma_{h\in \Lambda} \|a_{h}\|_{2}^{4}\ge \delta^{4} .
\end{align*}
Taking $Z_{0}\in L(\Lambda)\bar{\otimes}M\bar{\otimes}M $ to be the unique element of minimal $\|\cdot\|_{2}$-norm in  $\overline{\text{co}}^{\text{w.o}.}(\{(\Delta(x)\otimes x)_{132}(\Delta(x)\otimes x)^{*} | x \in \mathcal{G}\})$ (the weak operator closed convex hull of $\{(\Delta(x)\otimes x)_{132}(\Delta(x)\otimes x)^{*} | x \in \mathcal{G}\}$), the above inequality shows  $\tau(Z_{0}) > 0$. We also have $(\Delta(x)\otimes x)_{132}Z_{0}=Z_{0}(\Delta(x)\otimes x)$ for all $x\in \mathcal{G}$. Since we have $(\Delta(x)\otimes x)^{*}Z_{0}^{*}Z_{0}(\Delta(x)\otimes x) = Z_{0}^{*}Z_{0}$, and that the action $\{\mathrm{Ad} (\Delta(x)\otimes x)\}_{x \in \mathcal{G}}$ on $L(\Lambda)\bar{\otimes}M\bar{\otimes}M$ is ergodic, we see that $Z_{0}^{*}Z_{0} $ is a positive scalar. Thus, multiplying $Z_{0}$ by a positive scalar, we obtain a unitary $Z\in \mathcal{U}(L(\Lambda)\bar{\otimes}M\bar{\otimes}M )$ such that  $(\Delta(x)\otimes x)_{132}Z=Z(\Delta(x)\otimes x)$ holds for all $x\in \mathcal{G}$.

We show that this unitary $Z$ satisfies the second condition in the previous lemma. Using  $(\Delta(x)\otimes x)_{132}Z=Z(\Delta(x)\otimes x)$, a direct computation shows that \[Z_{134}Z_{123}Z_{124}^{*}\in \mathcal{U}(L(\Lambda)\bar{\otimes}M\bar{\otimes}M\bar{\otimes}M)\] commutes with $\Delta(x)_{14}(1\otimes x\otimes x \otimes 1) $ for all $x\in \mathcal{G}$. Now, the assumptions on the action of $\mathcal{G}$ shows  $Z_{134}Z_{123}Z_{124}^{*}\in\mathbb{T}$. Since $\tau (Z) >0$ and $\tau(Z_{134}^{*}Z_{124}) \ge 0$, this scalar must be $1$, and we have obtained $Z_{134}Z_{123} = Z_{124}$.

We now apply the previous lemma (with $L(\Lambda)=N$), and obtain a unitary $U\in \mathcal{U}(L(\Lambda)\bar{\otimes}M)$ such that $Z=U_{13}^{*}U_{12}$. Now, by $(\Delta(x)\otimes x)_{132}Z=Z(\Delta(x)\otimes x)$, we have $(U\Delta(x)U^{*}(1\otimes x^{*}))_{13}=(U\Delta(x)U^{*}(1\otimes x^{*}))_{12}$, which implies $U\Delta(x)U^{*}(1\otimes x^{*})\in M\otimes 1$ for all $x\in\mathcal{G}$. Writing $U\Delta(x)U^{*}(1\otimes x^{*})=w_{x}\otimes 1$ by $w_{x}\in \mathcal{U}(L(\Lambda))$, we have $U\Delta(x)U^{*}=w_{x}\otimes x$. This proves the proposition.
\end{proof}

We next show that  $U\Delta(x)U^{*} = w_{x}\otimes x$ for all $x\in \mathcal{G}$ (the conclusion of the previous proposition) with a few more assumptions roughly implies that $\mathcal{G}$ can be conjugated into $\mathcal{U}(A)\cdot\Lambda$.
\begin{lem}
Let $\sigma:\Lambda \curvearrowright A$ be a trace preserving action and $M=A\rtimes\Lambda$ the associated crossed product von Neumann algebra. Let $\mathcal{G} \subset \mathcal{U}(M)$ be a subgroup of the unitary group of $M$.
Assume that there exists a unitary  $U\in \mathcal{U}(M\bar{\otimes}M)$ and a unitary representation $\{w_{x}\}_{x\in \mathcal{G}}\in\mathcal{U}(M)$ of $\mathcal{G}$ such that the following conditions hold.
\begin{itemize}
\item $U\Delta(x)U^{*} = w_{x}\otimes x$ for all $x\in \mathcal{G}$. 
\item There exists a $\lambda\in \mathbb{T}$ such that $ ((1\otimes \Delta)(U^{*})(1\otimes U^{*}))_{213} = \lambda (1\otimes \Delta)(U^{*})(1\otimes U^{*})$.
\item $(\Delta\otimes 1)(U)(1\otimes \Delta)(U^{*})U_{23}^{*}\in \mathcal{U}(M\bar{\otimes}M\otimes 1)$.
\end{itemize}
Then there exists an $n\in  \mathbb{N}\cup \{\infty\}$,  a partial isometry $w\in M_{n,1}(M)$, a group homomorphism $\delta:\mathcal{G}\to \Lambda$, and a map  $x\in\mathcal{G}\mapsto a_{x}\in M_{n}(A)$ such that the following conditions hold.
\begin{itemize}
\item $w^{*}w=1, ww^{*} = e \in P(M_{n}(A))$.
\item $a_{x}a_{x}^{*} = e, a_{x}^{*}a_{x} = (\mathrm{Id}_{M_{n}}\otimes\sigma_{\delta(x)})(e)$  for all $x\in\mathcal{G}$.
\item $wxw^{*} = a_{x}(1_{n}\otimes v_{\delta(x)})$ for all $x\in\mathcal{G}$.
\end{itemize}  \label{newlem2}
\end{lem}
\begin{proof}
Set $W=\Sigma_{h\in\Lambda} v_{h}^{*}\otimes v_{h}e_{A}v_{h}^{*}\in \mathcal{U}(L(\Lambda)\bar{\otimes}\langle M,e_{A}\rangle)$ and $X=WU^{*}\in \mathcal{U}(M\bar{\otimes}\langle M,e_{A}\rangle)$. A direct computation shows $W^{*}(1\otimes y)W=\Delta(y)$ for all $y \in M$, which implies $(1\otimes x)X(1\otimes x^{*})=X(w_{x}\otimes 1)$ for $x\in \mathcal{G}$. We also have $(1\otimes Jv_{k}^{*}J)X (1\otimes Jv_{k}J) = (v_{k}\otimes 1)X$ for $k \in \Lambda$. Since $W_{23}^{*}W_{12}W_{13}=W_{12}W_{23}^{*}$, we have \[(\Delta\otimes 1)(X^{*})X_{13}X_{23}=(\Delta\otimes 1)(U)(1\otimes \Delta)(U^{*})U_{23}^{*}.\]

Writing the equality  $ ((1\otimes \Delta)(U^{*})(1\otimes U^{*}))_{213} = \lambda (1\otimes \Delta)(U^{*})(1\otimes U^{*})$ using $W$, we have $W_{13}^{*}U_{23}^{*}W_{13}U_{13}^{*} = \lambda W_{23}^{*}U_{13}^{*}W_{23}U_{23}^{*}$. Since $W_{23}$ and $W_{13}^{*}$ commute, we have $W_{23}U_{23}^{*}W_{13}U_{13}^{*} = \lambda W_{13}U_{13}^{*}W_{23}U_{23}^{*}$, that is, $X_{23}X_{13}=\lambda X_{13}X_{23}$. Take $y\in M$ such that  $b=(\tau\otimes 1)(X(y\otimes 1))$ is non zero. Multiplying by $y\otimes 1\otimes 1$ and applying $\tau \otimes 1\otimes 1$ to $X_{23}X_{13}=\lambda X_{13}X_{23}$, we have $X(1\otimes b )=\lambda (1\otimes b )X$. Taking adjoints, we also have $X(1\otimes b^{*} )=\lambda (1\otimes b^{*} )X$ and thus $bb^{*}=\lambda b^{*}b$. Since $b$ is non zero, we must have $\lambda=1$.

Let $D\subset \langle M,e_{A}\rangle$ be the von Neumann algebra generated by $\{ (\tau\otimes 1)(X(y\otimes 1))|y\in M\}$. Computing as in the above paragraph, and using $\lambda=1$, we see that any pair of elements in $\{ (\tau\otimes 1)(X(y\otimes 1))|y\in M\}\cup\{ (\tau\otimes 1)(X(y\otimes 1))|y\in M\}^{*}$ commute. Thus $D$ is an abelian von Neumann algebra. Moreover, by  $(1\otimes x)X(1\otimes x^{*})=X(w_{x}\otimes 1)$ for $x\in \mathcal{G}$ and $(1\otimes Jv_{k}^{*}J)X (1\otimes Jv_{k}J) = (v_{k}\otimes 1)X$  for $k \in \Lambda$, we see that $x\in\mathcal{G}$ and $Jv_{k}J$ for $k\in \Lambda$ normalizes $D$. Also, by definition, we have $X\in M\bar{\otimes}D$.

Take a measure space $(T,\mu)$ such that $D\cong L^{\infty}(T,\mu)$ holds. Viewing $X\in \mathcal{U}(M\bar{\otimes}D)$ as an $\mathcal{U}(M)$-valued function $X:t\in T \mapsto X_{t}\in\mathcal{U}(M)$, the condition  \[(\Delta\otimes 1)(X^{*})X_{13}X_{23}=(\Delta\otimes 1)(U)(1\otimes \Delta)(U^{*})U_{23}^{*}\in \mathcal{U}(M\bar{\otimes}M\otimes 1)\] shows that there is a conull subset $T_{0}\subset T$ such that $\Delta(X_{t}^{*})(X_{t}\otimes X_{t})$ is constant on $t\in T_{0}$. Take $t_{0}\in T_{0}$ and set $z=X_{t_{0}}\in \mathcal{U}(M)$. Then for $t\in T_{0}$, we have $\Delta(X_{t}^{*})(X_{t}\otimes X_{t})=\Delta(z^{*})(z\otimes z)$, thus $\Delta(zX_{t}^{*})=zX_{t}^{*}\otimes zX_{t}^{*}$, which implies $zX_{t}^{*} \in \{v_{h}\}_{h\in \Lambda}$. Thus, setting $Y=X(z^{*}\otimes 1)$, we have $X=Y(z\otimes 1)$ and that $Y\in\mathcal{U}(M\bar{\otimes}D)$ viewed as an $\mathcal{U}(M)$-valued function on $T$ takes values in $\{v_{h}\}_{h\in \Lambda}$ almost everywhere. From this, we can write $Y=\Sigma_{h\in \Lambda}v_{h}\otimes p_{h}$ with $p_{h}\in P(D)$ and $\Sigma_{h\in\Lambda}p_{h} = 1$. From $(1\otimes Jv_{k}^{*}J)X (1\otimes Jv_{k}J) = (v_{k}\otimes 1)X$ for $k \in \Lambda$, we have \[\Sigma_{h\in\Lambda}v_{h}\otimes Jv_{k}^{*}Jp_{h}Jv_{k}J=\Sigma_{h\in\Lambda}v_{kh}\otimes p_{h} = \Sigma_{h\in\Lambda} v_{h}\otimes p_{k^{-1}h},\] thus $Jv_{k}^{*}Jp_{h}Jv_{k}J=p_{k^{-1}h}$. Especially we have 
\begin{align*}
p_{k}=Jv_{k}Jp_{e}Jv_{k}^{*}J, \forall k\in \Lambda.
\end{align*}
Since $ \Sigma_{h\in\Lambda}p_{h}=1$, $p_{h}$ is non zero for every $h \in \Lambda$.

For each $x\in\mathcal{G}$, we have $(1\otimes x)X(1\otimes x^{*})=X(w_{x}\otimes 1)$ and thus $(1\otimes x)Y(1\otimes x^{*})=Y(zw_{x}z^{*}\otimes 1)$. Take $k\in\Lambda$ such that $p_{e}xp_{k}x^{*}\in P(D)$ is non zero. Multiplying the above equality by  $1\otimes p_{e}xp_{k}x^{*}\in P(D)$ , we obtain $v_{k}\otimes p_{e}xp_{k}x^{*}= zw_{x}z^{*}\otimes p_{e}xp_{k}x^{*}$, so $zw_{x}z^{*}=v_{k}$. This shows that we have 
\begin{align}
\{zw_{x}z^{*}\}_{x\in\mathcal{G}}\subset \{v_{h}\}_{h\in\Lambda},  \label{incl}
\end{align}
so there exists a group homomorphism $\delta:\mathcal{G}\to \Lambda$ such that $zw_{x}z^{*}=v_{\delta(x)}$ holds for all $x\in\mathcal{G}$. From $(1\otimes x)Y(1\otimes x^{*})=Y(zw_{x}z^{*}\otimes 1)$, we have \[\Sigma_{h\in\Lambda}v_{h}\otimes xp_{h}x^{*}=\Sigma_{h\in\Lambda}v_{h\delta(x)}\otimes p_{h} = \Sigma_{h\in\Lambda} v_{h}\otimes p_{h\delta(x)^{-1}},\] which implies
\begin{align*}
xp_{h}x^{*}=p_{h\delta(x)^{-1}}, \forall h\in\Lambda, \forall x\in \mathcal{G}.
\end{align*}

Consider the operator valued weight $E:\langle M,e_{A}\rangle\to M$ which satisfies $E(ve_{A}w)=vw$ for $v,w\in M$. We have  $E(p_{e})=\Sigma_{h\in\Lambda}Jv_{h}^{*}Jp_{e}Jv_{h}J=\Sigma_{h\in\Lambda}p_{h}=1$. Since the central support of $e_{A}$ in $\langle M,e_{A}\rangle$ is $1$, we can take a family of partial isometries $\{V_{i}\}_{i\in I}\in \langle M,e_{A}\rangle$ with $I=\{1,  2,\dots n\}$ for some $n\in \mathbb{N}\cup \{\infty\}$ such that $V_{i}^{*}V_{i} \le e_{A}$ and $\Sigma_{i} V_{i}V_{i}^{*}=p_{e}$ holds. Setting $w_{i}=V_{i}\hat{1}\in  L^{2}(M)$, we have $w_{i}e_{A}=V_{i}e_{A}$ and thus $V_{i}V_{i}^{*}=w_{i}e_{A}w_{i}^{*}$ and $\Sigma_{i}w_{i}e_{A}w_{i}^{*}=p_{e}$. Applying $E$ to this equality, we obtain $\Sigma_{i} w_{i}w_{i}^{*}=1$. Since $\{w_{i}e_{A}w_{i}^{*}\}_{i}$ is an orthogonal family of projections, we have $E_{A}(w_{i}^{*}w_{i})\in P(A)$ for all $i$ and $E_{A}(w_{i}^{*}w_{j})=0$ for  all $i\ne j$. Also, since $\{p_{h}\}_{h}$ is an orthogonal family. we see that $Jv_{h}Jw_{i}e_{A}w_{i}^{*}Jv_{h}^{*}J =w_{i}v_{h}^{*}e_{A}v_{h}w_{i}^{*}$ and $w_{j}e_{A}w_{j}^{*}$ are orthogonal projections for any $h \in \Lambda\setminus\{e\}$ and $i,j$. Thus we have $E_{A}(w_{j}^{*}w_{i}v_{h}^{*}) =0$ for all  $h \in \Lambda\setminus\{e\}$ and $i,j$. Altogether, we see that $w_{i}^{*}w_{i}\in P(A)$ for all $i$ and $w_{i}^{*}w_{j}=0$ for all $i\ne j$.

Now, set $w= \begin{pmatrix} w_{1},w_{2},\dots w_{n}\end{pmatrix}^{*}\in M_{n,1}(M)$. From the computations in the above paragraph, we have $w^{*}w=\Sigma_{i}w_{i}w_{i}^{*}=1$ and $ww^{*}=e\in P(M_{n}(A))$ where $e$ has $w_{i}^{*}w_{i}$ as its $(i,i)$-th entry for all $i$ and zero elsewhere. We also have $w^{*}(1_{n}\otimes e_{A})w=\Sigma_{i}w_{i}e_{A}w_{i}^{*}=p_{e}$. 

Coming back to the equality  $xp_{e}x^{*}=p_{\delta(x)^{-1}}$ for all $x\in\mathcal{G}$, we have 
\begin{align*}
xw^{*}(1_{n}\otimes e_{A})wx^{*} &=Jv_{\delta(x)}^{*}Jp_{e}Jv_{\delta(x)}J \\
&=w^{*}(1_{n}\otimes Jv_{\delta(x)}^{*}J)(1_{n}\otimes e_{A})(1_{n}\otimes Jv_{\delta(x)}J)w \\
&= w^{*}(1_{n}\otimes v_{\delta(x)}e_{A}v_{\delta(x)}^{*})w.
\end{align*}
From this equality, we see that $(1_{n}\otimes v_{\delta(x)}^{*})wxw^{*}$ commute with $1_{n}\otimes e_{A}$ for all  $x\in\mathcal{G}$. Since $M \cap \{e_{A}\}' = A$, this shows  $(1_{n}\otimes v_{\delta(x)}^{*})wxw^{*} \in M_{n}(A)$, and hence $wxw^{*}\in M_{n}(A)\cdot (1_{n}\otimes v_{\delta(x)})$ for all $x\in\mathcal{G}$. This proves the lemma.
\end{proof}

\begin{rem}
By the proof, we can take $e\in P(M_{n}(A))$ in the conclusion of the above lemma to have non zero entries only on the diagonal.
Moreover, if we have that the action $\{\mathrm{Ad}v_{h}\}_{h\in \delta(\mathcal{G})}$ on $Z(A)$ (the center of $A$) is ergodic, we can take $n$ to be one and $w\in\mathcal{U}(M), a_{x}\in \mathcal{U}(A)$ for all $x\in \mathcal{G}$.

Indeed, assume  that the action $\{\mathrm{Ad}v_{h}\}_{h\in \delta(\mathcal{G})}$ on $Z(A)$  is ergodic and $e$ is diagonal as above. Denote by $e_{i}\in P(A)$ the $(i,i)$-th entry of $e$. Consider the center-valued tracial weight $\mathrm{Tr}\otimes E_{Z(A)}$ on $M_{n}(A)$. Since we have $e=wxw^{*}ewx^{*}w^{*}=a_{x}(1_{n}\otimes v_{\delta(x)})e(1_{n}\otimes v_{\delta(x)})^{*}a_{x}^{*}$, applying  $\mathrm{Tr}\otimes E_{Z(A)}$, we obtain 
\begin{align*}(\mathrm{Tr}\otimes E_{Z(A)})(e)
&=(\mathrm{Tr}\otimes E_{Z(A)})(a_{x}(1_{n}\otimes v_{\delta(x)})e(1_{n}\otimes v_{\delta(x)})^{*}a_{x}^{*})\\
&=(\mathrm{Tr}\otimes E_{Z(A)})((1_{n}\otimes v_{\delta(x)})e(1_{n}\otimes v_{\delta(x)})^{*})\\
&=v_{\delta(x)}(\mathrm{Tr}\otimes E_{Z(A)})(e)v_{\delta(x)}^{*}.
\end{align*}
Hence $(\mathrm{Tr}\otimes E_{Z(A)})(e)$ is a $\delta(\mathcal{G})$-invariant operator in $L^{1}(Z(A))$ with trace one. By the ergodicity of $\delta(\mathcal{G})$, we obtain  $(\mathrm{Tr}\otimes E_{Z(A)})(e)=1$. This shows that we can take partial isometries $v_{i}\in A$ such that $v_{i}^{*}v_{i}=e_{i}$ and $\Sigma_{i}v_{i}v_{i}^{*}=1$ holds. Setting $v=\begin{pmatrix} v_{1},v_{2},\dots v_{n}\end{pmatrix}\in M_{1,n}(A)$, we have $vv^{*}=1, v^{*}v=e$ and \[vwxw^{*}v^{*}=va_{x}(1_{n}\otimes v_{\delta(x)})v^{*}=(va_{x}(1_{n}\otimes v_{\delta(x)})v^{*}v_{\delta(x)}^{*})v_{\delta(x)}\] for all $x\in \mathcal{G}$. Thus, setting $\tilde{w}=vw\in \mathcal{U}(M)$ and $\tilde{a_{x}}=va_{x}(1_{n}\otimes v_{\delta(x)})v^{*}v_{\delta(x)}^{*}\in \mathcal{U}(A)$, we have $\tilde{w}x\tilde{w}^{*}=\tilde{a_{x}}v_{\delta(x)}$ for all  $x\in \mathcal{G}$. \label{rem1}
\end{rem}
We can now prove the following proposition.
\begin{prop}
Let $\sigma:\Lambda \curvearrowright A$ be a trace preserving action and $M=A\rtimes\Lambda$ the associated crossed product $\mathrm{II}_{1}$ factor. Let $\mathcal{G} \subset \mathcal{U}(M)$ be a subgroup of the unitary group of $M$. Assume that the following conditions hold.
\begin{itemize}
\item The unitary representation $\{\mathrm{Ad}\Delta(x)\}_{x\in \mathcal{G}}$ on $L^{2}(L(\Lambda)\bar{\otimes}M) \ominus \mathbb{C}$ is weakly mixing.
\item   $\inf_{x\in \mathcal{G}}\sup_{h\in\Lambda}\|E_{A}(xv_{h}^{*})\|_{2} >0$.
\end{itemize}
Then there exists an $n\in  \mathbb{N}\cup \{\infty\}$,  a partial isometry $w\in M_{n,1}(M)$, a group homomorphism $\delta:\mathcal{G}\to \Lambda$, and a map  $x\in\mathcal{G}\mapsto a_{x}\in M_{n}(A)$ such that the following conditions hold.
\begin{itemize}
\item $w^{*}w=1, ww^{*} = e \in P(M_{n}(A))$.
\item $a_{x}a_{x}^{*} = e, a_{x}^{*}a_{x} = (\mathrm{Id}_{M_{n}}\otimes\sigma_{\delta(x)})(e)$  for all $x\in\mathcal{G}$.
\item $wxw^{*} = a_{x}(1_{n}\otimes v_{\delta(x)})$ for all $x\in\mathcal{G}$.
\end{itemize}  \label{mainprop}
\end{prop}
\begin{proof} First, we apply Proposition \ref{prop1} and obtain  a unitary $U\in \mathcal{U}(L(\Lambda)\bar{\otimes}M)$ and a unitary representation $\{w_{x}\}_{x\in \mathcal{G}}\in \mathcal{U}(L(\Lambda))$ of $\mathcal{G}$ with $U\Delta(x)U^{*} = w_{x}\otimes x$ for all $x\in \mathcal{G}$.

We next check that this $U\in \mathcal{U}(L(\Lambda)\bar{\otimes}M)$ satisfies the conditions in Lemma \ref{newlem2}.
A direct computation shows  \[((1\otimes \Delta)\Delta)(x)(1\otimes \Delta)(U^{*})(1\otimes U^{*}) =(1\otimes \Delta)(U^{*})(1\otimes U^{*}) (w_{x}\otimes w_{x}\otimes x)\] for $x \in \mathcal{G}$, which also implies  \[((1\otimes \Delta)\Delta)(x)((1\otimes \Delta)(U^{*})(1\otimes U^{*}))_{213} =((1\otimes \Delta)(U^{*})(1\otimes U^{*}))_{213} (w_{x}\otimes w_{x}\otimes x)\] for $x \in \mathcal{G}$.  Since $\{\mathrm{Ad}\Delta(x)\}_{x\in \mathcal{G}}$ is weakly mixing on $L^{2}(L(\Lambda)\bar{\otimes}M) \ominus \mathbb{C}$, $\{\mathrm{Ad} (w_{x}\otimes x)\}_{x\in\mathcal{G}}$ is weakly mixing on $L^{2}(L(\Lambda)\bar{\otimes}M) \ominus \mathbb{C}$ and thus $\{\mathrm{Ad} w_{x}\}_{x\in\mathcal{G}}$ is ergodic on $L(\Lambda)$. This shows that $\{\mathrm{Ad} (w_{x}\otimes w_{x}\otimes x)\}_{x\in \mathcal{G}}$ is ergodic on $L(\Lambda)\bar{\otimes}L(\Lambda)\bar{\otimes}M$, which implies that there exists a $\lambda\in \mathbb{T}$ such that \[ ((1\otimes \Delta)(U^{*})(1\otimes U^{*}))_{213} = \lambda (1\otimes \Delta)(U^{*})(1\otimes U^{*})\] holds. 

Similarly, we see that \[(\Delta(w_{x})\otimes x)(\Delta\otimes 1)(U)(1\otimes \Delta)(U^{*})U_{23}^{*}=(\Delta\otimes 1)(U)(1\otimes \Delta)(U^{*})U_{23}^{*}(w_{x}\otimes w_{x}\otimes x)\] holds for all $x\in\mathcal{G}$. Since $\{\mathrm{Ad} x\}_{x \in \mathcal{G}}$ is weakly mixing on $L^{2}(M)\ominus \mathbb{C}$, this equality implies $(\Delta\otimes 1)(U)(1\otimes \Delta)(U^{*})U_{23}^{*}\in L(\Lambda)\bar{\otimes}L(\Lambda)\otimes 1$.

Now we can apply  Lemma \ref{newlem2} and obtain a partial isometry $w\in M_{n,1}(M)$, a group homomorphism $\delta:\mathcal{G}\to \Lambda$, and a map  $x\in\mathcal{G}\mapsto a_{x}\in M_{\infty}(A)$ such that required properties hold. 
\end{proof}

As a particular case of the previous proposition, we obtain the following theorem.

\begin{thm}
Let $M=A\rtimes\Lambda$ be a crossed product $\mathrm{II}_{1}$ factor with $\Lambda$ an i.c.c. group. Let $\mathcal{G} \subset \mathcal{U}(M)$ be a subgroup of the unitary group of $M$ which satisfies the following properties.
\begin{enumerate}
\item  The unitary representation $\{\mathrm{Ad} x\}_{x \in \mathcal{G}}$ on $L^{2}(M)\ominus \mathbb{C}$ is weakly mixing.
\item $\mathcal{G}''\nprec_{M}A\rtimes\Lambda_{0}$ for every infinite index subgroup $\Lambda_{0} \subset \Lambda$.
\item $\inf_{x\in \mathcal{G}}\sup_{h\in\Lambda}\|E_{A}(xv_{h}^{*})\|_{2} >0$.
\end{enumerate}
Then there exists an $n\in  \mathbb{N}\cup \{\infty\}$,  a partial isometry $w\in M_{n,1}(M)$, a group homomorphism $\delta:\mathcal{G}\to \Lambda$, and a map  $x\in\mathcal{G}\mapsto a_{x}\in M_{n}(A)$ such that the following conditions hold.
\begin{itemize}
\item $w^{*}w=1, ww^{*} = e \in P(M_{n}(A))$.
\item $a_{x}a_{x}^{*} = e, a_{x}^{*}a_{x} = (\mathrm{Id}_{M_{n}}\otimes\sigma_{\delta(x)})(e)$  for all $x\in\mathcal{G}$.
\item $wxw^{*} = a_{x}(1_{n}\otimes v_{\delta(x)})$ for all $x\in\mathcal{G}$.
\end{itemize}
If in addition we assume that the action of $\Lambda$ on $L^{2}(Z(A))\ominus \mathbb{C}$ is weakly mixing, there is a unitary
$w\in \mathcal{U}(M)$ such that $w\mathcal{G}w^{*}\subset \{av_{h}| a\in \mathcal{U}(A), h \in \Lambda\}$ holds. \label{theorem}
\end{thm}
\begin{proof} From the first and second items and Lemma \ref{lem0}, the unitary representation $\{\mathrm{Ad}\Delta(x)\}_{x\in \mathcal{G}}$ on $L^{2}(L(\Lambda)\bar{\otimes}M) \ominus \mathbb{C}$ is weakly mixing (note that subgroups of the form $C_{\Lambda}(h)$ for $h\in \Lambda \setminus \{e\}$ has infinite index in $\Lambda$ since $\Lambda$ is an i.c.c. group). Thus we can apply the previous proposition.

Assume now that the action of  $\Lambda$ on $L^{2}(Z(A))\ominus \mathbb{C}$ is weakly mixing,  From $wxw^{*} = a_{x}(1_{n}\otimes v_{\delta(x)})$ for all $x\in\mathcal{G}$, we easily see that  $\mathcal{G}''\prec_{M}A\rtimes\delta(\mathcal{G})$ holds, and by our second item, $\delta(\mathcal{G})$ has finite index in $\Lambda$. Combining with our assumption that the  action of $\Lambda$ on $L^{2}(Z(A))\ominus \mathbb{C}$ is weakly mixing, we see that the action of $ \delta(\mathcal{G})$ on $Z(A)$ is ergodic. Thus we can apply  Remark \ref{rem1} and obtain the desired unitary $w\in \mathcal{U}(M)$.
\end{proof}

\begin{proof}[Proof of Theorem \ref{mainthm}]
Set $\mathcal{G}=\{bu_{g}|b\in \mathcal{U}(B), g\in \Gamma\}  \subset \mathcal{U}(M)$. If we have $M =A\rtimes\Lambda= B\times\Gamma$, we automatically have $\mathcal{G}''=M\nprec_{M}A\rtimes\Lambda_{0}$ for every infinite index subgroup $\Lambda_{0} \subset \Lambda$. Thus we can apply the above theorem.
\end{proof}

\begin{rem}
Taking $A=\mathbb{C}$ and $\mathcal{G}=\Gamma$ for a discrete group satisfying $L(\Gamma)\subset L(\Lambda)$ in Theorem \ref{theorem}, we obtain \cite[Theorem 3.1]{ioana2013class}. See also \cite[Theorem 4.1]{krogager2017class}.
\end{rem}

Using  Lemma \ref{newlem2}, we also give a condition when we can deduce the existence of a unitary  $w\in \mathcal{U}(M)$ such that $w\mathcal{G}w^{*}\subset \mathbb{T}\{v_{h}\}_{h\in \Lambda}$ holds. We closely follow the proof of \cite[Theorem 3.1]{ioana2013class}. 

\begin{thm}
Let $M=A\rtimes\Lambda$ be a crossed product $\mathrm{II}_{1}$ factor with $\Lambda$ an i.c.c. group and $\mathcal{G} \subset \mathcal{U}(M)$ a subgroup of the unitary group of $M$ which satisfies the following properties.
\begin{enumerate}
\item The unitary representation $\{\mathrm{Ad}\Delta(x)\}_{x\in \mathcal{G}}$ on $L^{2}(M\bar{\otimes}M) \ominus \mathbb{C}$ is weakly mixing.
\item There exists a finite set $F\subset A$ such that $\inf_{x\in \mathcal{G}}\sup_{h\in\Lambda}\Sigma_{z\in F}|\tau(E_{A}(xv_{h}^{*})z)|^{2} >0$ holds.
\end{enumerate}
 Then there exists a unitary $w\in \mathcal{U}(M)$ such that $w\mathcal{G}w^{*}\subset \mathbb{T}\{v_{h}\}_{h\in \Lambda}$ holds.
\end{thm}
\begin{proof} We may assume that $F\subset A$ is an orthonormal system with respect to $\tau$. Set $\delta=\inf_{x\in \mathcal{G}}\sup_{h\in\Lambda}\Sigma_{z\in F}|\tau(E_{A}(xv_{h}^{*})z)|^{2} $. For any $x=\Sigma_{h\in \Lambda}a_{h}v_{h}\in\mathcal{G}$, we have
\begin{align*}
\Sigma_{z\in F}\tau((\Delta(x)\otimes x)(x\otimes \Delta(x))^{*}(z^{*}\otimes z \otimes 1))&=
\Sigma_{z\in F}\Sigma_{h\in\Lambda}|\tau(a_{h}z)|^{2}\tau(a_{h}a_{h}^{*}).
\end{align*}
Since $F$ is an orthonormal system in $A$, we have $\Sigma_{z\in F}|\tau(a_{h}z)|^{2}\le \tau(a_{h}a_{h}^{*})$ for all $h\in \Lambda$. Thus we can continue as 
\begin{align*}
\Sigma_{z\in F}\Sigma_{h\in\Lambda}|\tau(a_{h}z)|^{2}\tau(a_{h}a_{h}^{*}) &\ge
\Sigma_{h\in\Lambda} (\Sigma_{z\in F}|\tau(a_{h}z)|^{2})^{2} \\
&\ge \delta^{2}
\end{align*}
for all $x \in \mathcal{G}$. Taking $Z\in M\bar{\otimes}M\bar{\otimes}M$ to be the  $\|\cdot\|_{2}$-minimum element in  $\overline{\text{co}}^{\text{w.o}.}\{\Sigma_{z\in F}(\Delta(x)\otimes x)(x\otimes \Delta(x))^{*}|x\in \mathcal{G}\}$, $Z$ is non zero and $(\Delta(x)\otimes x)Z=Z(x\otimes \Delta(x))$ holds for all $x \in \mathcal{G}$.  Hence $Z^{*}Z$ is a positive scalar and so we may assume that $Z$ is a unitary. Setting $Y=(Z\otimes 1)(1\otimes Z) \in \mathcal{U}(M\bar{\otimes}M\bar{\otimes}M\bar{\otimes}M)$, we have $(\Delta(x)\otimes x\otimes x)Y=Y(x\otimes x\otimes \Delta(x))$ for all $x\in \mathcal{G}$. This shows that the unitary representation $\xi\mapsto (x\otimes x)\xi\Delta(x)^{*}$ of $\mathcal{G}$ on $L^{2}(M\bar{\otimes}M)$ is not weakly mixing. Thus we can take an irreducible finite dimensional unitary representation $\eta:\mathcal{G}\to\mathcal{U}(\mathbb{C}^{n})$ and a non-zero vector $\xi \in \mathbb{C}^{n}\otimes L^{2}(M\bar{\otimes}M)$ with $(\eta(x)\otimes x \otimes x)\xi =\xi\Delta(x)$ for all $x \in \mathcal{G}$. Since $\{\mathrm{Ad}(x\otimes x)\}_{x\in \mathcal{G}}$ is weakly mixing on $L^{2}(M\bar{\otimes}M)\ominus\mathbb{C}$ and $\eta$ is irreducible, $\xi\xi^{*}$ is a scalar. Hence $n=1$ and we have found a unitary $U\in\mathcal{U}( M\bar{\otimes}M)$ and a character $\gamma:\mathcal{G}\to\mathbb{T}$ such that $\gamma(x)(x\otimes x)U=U\Delta(x)$ holds for all  $x \in \mathcal{G}$. As in the proof of Proposition \ref{mainprop}, we see that this $U\in\mathcal{U}( M\bar{\otimes}M)$ satisfies the assumptions of  Lemma \ref{newlem2}.

Now we can start reading the proof of  Lemma \ref{newlem2} with $w_{x}=\gamma(x)x$ for all $x \in \mathcal{G}$. At (\ref{incl}), we obtain $z\in\mathcal{U}(M)$ such that $\{\gamma(x)zxz^{*}\}_{x\in\mathcal{G}}\subset \{v_{h}\}_{h\in\Lambda}$ holds. This is enough for this theorem.
\end{proof}

We next prove an easy lemma which shows that the property \[\inf_{x\in \mathcal{G}}\sup_{h\in\Lambda}\|E_{A}(xv_{h}^{*})\|_{2} >0\] for $\mathcal{G}$ is invariant under unitary conjugacy.

\begin{lem}
Let $M=A\rtimes\Lambda$ be a crossed product von Neumann algebra and $\mathcal{G}\subset\mathcal{U}(M)$ a subgroup. Then for any $w\in\mathcal{U}(M)$, we have $\inf_{x\in \mathcal{G}}\sup_{h\in\Lambda}\|E_{A}(xv_{h}^{*})\|_{2} >0$ if and only if  $\inf_{x\in w\mathcal{G}w^{*}}\sup_{h\in\Lambda}\|E_{A}(xv_{h}^{*})\|_{2} >0$. \label{uclem}
\end{lem}
\begin{proof}
We show that if $(x_{n})_{n}\subset \mathcal{G}$ is a sequence such that $\sup_{h\in\Lambda}\|E_{A}(x_{n}v_{h}^{*})\|_{2}\to 0$ as $n\to\infty$, we have $\sup_{h\in\Lambda}\|E_{A}(wx_{n}w^{*}v_{h}^{*})\|_{2}\to 0$, which clearly implies the lemma.

Take any $\epsilon>0$. By the Kaplansky theorem, we can take a finite set $F\subset\Lambda$ and a $w_{0}=\Sigma_{h\in F}b_{h}v_{h}\in M$ with $\|w-w_{0}\|_{2}<\epsilon$ and $\|w_{0}\|\le 1$.  Writing $x_{n}=\Sigma_{h\in\Lambda}a_{h}v_{h}$, we have
\begin{align*}
w_{0}x_{n}w_{0}^{*} &= \Sigma_{h,k,l\in\Lambda}b_{k}v_{k}a_{h}v_{h}v_{l}^{*}b_{l}^{*} \\
&= \Sigma_{h}(\Sigma_{k,l} b_{k}\sigma_{k}(a_{l})\sigma_{h}(b_{h^{-1}kl}^{*})) v_{h},
\end{align*}
and thus
\begin{align*}
\|E_{A}(w_{0}x_{n}w_{0}^{*}v_{h}^{*})\|_{2} &= \|\Sigma_{k,l} b_{k}\sigma_{k}(a_{l})\sigma_{h}(b_{h^{-1}kl}^{*})\|_{2} \\
&\le \Sigma_{k\in F}\Sigma_{l\in k^{-1}hF}\|b_{k}\|_{\infty}\cdot\|b_{h^{-1}kl}\|_{\infty}\cdot\|a_{l}\|_{2} \\
&\le \Sigma_{k\in F}\Sigma_{l\in k^{-1}hF}\|a_{l}\|_{2} \\
&\le |F|^{2}\cdot \sup_{h\in\Lambda}\|E_{A}(x_{n}v_{h}^{*})\|_{2},
\end{align*}
for all $h\in\Lambda$. Hence  $\sup_{h\in\Lambda}\|E_{A}(w_{0}x_{n}w_{0}^{*}v_{h}^{*})\|_{2}$ tends to $0$ as $n\to \infty$. Since $\|E_{A}(w_{0}x_{n}w_{0}^{*}v_{h}^{*})-E_{A}(wx_{n}w^{*}v_{h}^{*})\|_{2}\le2\epsilon$ for  all $h\in\Lambda$, we have $\limsup_{n}(\sup_{h\in\Lambda}\|E_{A}(wx_{n}w^{*}v_{h}^{*})\|_{2})\le 2\epsilon$. Since this holds for any $\epsilon>0$, we have $\sup_{h\in\Lambda}\|E_{A}(wx_{n}w^{*}v_{h}^{*})\|_{2}\to 0$.
\end{proof}

We end this section with an application of Lemma \ref{newlem}.
\begin{lem} \label{lastlem}
Let $N,M,L$ be tracial von Neumann algebras. For a unitary $Z\in\mathcal{U}(N\bar{\otimes}M\bar{\otimes}L)$, the following conditions are equivalent.
\begin{enumerate}
\item $Z_{124}^{*}Z_{123} \in 1\bar{\otimes} M\bar{\otimes}L\bar{\otimes} L$.
\item There exist unitaries $a\in\mathcal{U}(N\bar{\otimes}M)$ and $b\in\mathcal{U}(M\bar{\otimes}L)$ such that $Z=a_{12}b_{23}$ holds.
\end{enumerate}
\end{lem}
\begin{proof}
It is easy to see that the second condition implies the first condition, so we prove the opposite direction. Let  $W\in\mathcal{U}(M\bar{\otimes}L\bar{\otimes} L)$ be the unitary satisfying  $Z_{124}^{*}Z_{123} = 1\otimes W$. We can  directly check that this $W$ satisfies $W_{134}W_{123} = W_{124}$. Using Lemma \ref{newlem}, we obtain a unitary $b\in\mathcal{U}(M\bar{\otimes}L)$ such that $W=b_{13}^{*}b_{12}$ holds. Since $Z_{124}^{*}Z_{123} = 1\otimes W = b_{24}^{*}b_{23}$, we obtain $Z_{123}b_{23}^{*}=Z_{124}b_{24}^{*}$, which implies $Z_{123}b_{23}^{*}\in N\bar{\otimes}M\bar{\otimes} 1 \bar{\otimes} 1$. Taking $a\in\mathcal{U}(N\bar{\otimes}M)$ to be the unitary satisfying $Z_{123}b_{23}^{*}=a_{12}$, we have $Z=a_{12}b_{23}$.
\end{proof}

\section{APPLICATIONS TO UNITARY COCYCLES }
In this section, we show that the results in the previous section can be useful to the study of cocycles. 
We first review some terminology on cocycles from  \cite{popa2007cocycle}.
Let $N$ be a tracial von Neumann algebra and $\Lambda$ a discrete group with an action $\sigma:\Lambda \curvearrowright N$. A \textit{cocycle} for $\sigma$ is map $w:\Lambda \to \mathcal{U}(N)$ satisfying 
\begin{align}
w_{h}\sigma_{h}(w_{k}) = w_{hk} \label{uc1}
\end{align}
for all $h,k \in \Lambda$. This condition is equivalent to the fact that the map $h\in\Lambda \mapsto w_{h}v_{h}\in \mathcal{U}(N\rtimes\Lambda)$ is a group homomorphism. Two cocycles $w,w'$ are said to be \textit{cohomologous} if there exists a unitary $u\in\mathcal{U}(N)$ such that 
\begin{align}
u^{*}w_{h}\sigma_{h}(u)=w'_{h} \label{uc2}
\end{align}
holds for all $h\in\Lambda$.

Next, let $\sigma$ be a probabilty measure preserving action of a countable discrete group $\Lambda$ on a standard probability space $(X,\mu)$ and $\mathcal{V}$ a Polish group. A \textit{measurable cocycle} for $\sigma$ with values in $\mathcal{V}$ is a measurable map $w:X\times \Lambda \to \mathcal{V}$ such that for all $h_{1},h_{2}\in \Lambda$, 
\begin{align}
w(t,h_{1})w(h_{1}^{-1}t,h_{2})=w(t,h_{1}h_{2}) \label{mc1}
\end{align}
holds for $\mu$-almost everywhere $t\in X$. Two measurable cocycles $w,w'$ are said to be \textit{cohomologous} if there exists a measurable map $u:X\to \mathcal{V}$ such that for all $h\in\Lambda$,
\begin{align}
w'(t,h)=u(t)^{-1}w(t,h)u(h^{-1}t) \label{mc2}
\end{align}
holds for $\mu$-alomost everywhere $t\in X$.

The next lemma is fundamental. Note that  any measurable map $X\to \mathcal{V}$ can be naturally regarded as an element of $\mathcal{U}(L^{\infty}(X,\mu)\bar{\otimes}B)$. 

\begin{lem}[\cite{popa2007cocycle}, Lemma 2.3]
Let  $\sigma:\Lambda\curvearrowright (X,\mu)$ be a probability measure preserving action of a countable discrete group $\Lambda$ on a standard probability space $(X,\mu)$ and $\mathcal{V}$ a Polish group. Assume that $\mathcal{V}$ is a closed subgroup of the unitary group $\mathcal{U}(B)$ of a tracial von Neumann algebra $B$. Then, a map $w:X\times \Lambda \to \mathcal{V}$ is a measurable cocycle for $\sigma:\Lambda\curvearrowright (X,\mu)$ if and only if $h\in\Lambda\mapsto w_{h}=w(\cdot,h)\in \mathcal{U}(L^{\infty}(X,\mu)\bar{\otimes}B)$ is a cocycle for the action $\sigma\otimes id_{B}:\Lambda\curvearrowright L^{\infty}(X,\mu)\bar{\otimes}B$. Moreover,  two measurable cocycles $w,w'$ and  a measurable map $u:X\to \mathcal{V}$ satisfies equation (\ref{mc2}) if and only if equation (\ref{uc2}) holds when viewed as elements of $\mathcal{U}(L^{\infty}(X,\mu)\bar{\otimes}B)$. \label{ucmc}
\end{lem}

\begin{rem}
For a standard probability space $(X,\mu)$, we view $L^{\infty}(X,\mu)$ as a tracial von Neumann algebra with trace given by integration with respect to $\mu$. When  $\sigma:\Lambda\curvearrowright (X,\mu)$ is a probability measure preserving action, it naturally induces a trace preserving action $\Lambda\curvearrowright L^{\infty}(X,\mu)$, which is denoted by the same symbol $\sigma$.
\end{rem}

We now give a proof of the following theorem by Popa using the techniques from the previous section.
\begin{thm}[\cite{popa2007cocycle}, Theorem 3.1, \cite{furman2007popa}, Theorem 3.4 ]
Let $A, B, C$ be tracial von Neumann algebras and $\Lambda$ a countable discrete group. Let $\sigma^{A}: \Lambda\curvearrowright A, \sigma^{B}: \Lambda\curvearrowright B,  \sigma^{C}: \Lambda\curvearrowright C $ be trace peserving actions.
Assume that $\sigma^{A}$ is weakly mixing. Let $w:\Lambda \to \mathcal{U}(A\bar{\otimes}B)$ and $\tilde{w}:\Lambda\to \mathcal{U}(B\bar{\otimes}C)$ be cocycles for the actions $\sigma^{A}\otimes \sigma^{B}$ and $\sigma^{B}\otimes\sigma^{C}$, respectively. 
Assume that the cocycles $w, \tilde{w}$ are cohomologous in $\mathcal{U}(A\bar{\otimes}B\bar{\otimes}C)$ as cocycles for the action $\sigma^{A}\otimes\sigma^{B}\otimes\sigma^{C}$. Then, the following statements hold true:
\begin{itemize}
\item $w$ is cohomologous (as a cocycle for the action $\sigma^{A}\otimes \sigma^{B}$) to a cocycle which takes values in $\mathcal{U}(B)$.
\item $\tilde{w}$ is cohomologous (as a cocycle for the action  $\sigma^{B}\otimes\sigma^{C}$) to a cocycle which takes values in $\mathcal{U}(B)$.
\end{itemize}
\end{thm}
\begin{proof} Consider the action  $\sigma=\sigma^{A}\otimes\sigma^{B}\otimes\sigma^{C}\otimes\sigma^{C}$ on  $A\bar{\otimes}B\bar{\otimes}C\bar{\otimes}C$. By our assumption, there is a unitary $Z\in \mathcal{U}(A\bar{\otimes}B\bar{\otimes}C)$ such that $(w_h)_{12}\sigma_h(Z) = Z (\tilde{w}_h)_{23}$ holds for all $h\in \Lambda$.
Thus we have \[Z_{124}^{*}Z_{123}=(\tilde{w}_h)_{24}\sigma_{h}(Z_{124}^{*}Z_{123})(\tilde{w}_h^{*})_{23}\]  for all $h\in \Lambda$. Since $\sigma^{A}$ is weakly mixing, we have $Z_{124}^{*}Z_{123}\in 1\bar{\otimes}B\bar{\otimes}C\bar{\otimes}C$. By Lemma \ref{lastlem}, we can take unitaries $u\in\mathcal{U}(A\bar{\otimes}B)$ and $v\in\mathcal{U}(B\bar{\otimes}C)$ with $Z=u_{12}v_{23}$. Now, by $w_h\sigma_h(Z) = Z \tilde{w}_h$, we have \[u_{12}^{*}w_h\sigma_h(u_{12})=v_{23}\tilde{w}_h\sigma(v_{23}^{*})\] for all $h\in \Lambda$. Since the left-hand side is in $A\bar{\otimes}B$, and the right-hand side is in $B\bar{\otimes}C$, we see that both sides are in $B$. This proves the theorem.
\end{proof}


\section{ACKNOWLEDGEMENTS} 
The author would like to thank his supervisor, Professor Yasuyuki Kawahigashi for his helpful comments and continuing support. He is also grateful to Yusuke Isono for some comments on this paper. The author is supported by  JSPS
KAKENHI Grant Number JP22J10118 and JST CREST program JPMJCR18T6.

\bibliographystyle{amsalpha}
\bibliography{myref(1)}

\end{document}